%% file: main.tex
\documentclass[letterpaper]{article}

\usepackage{packages-styles/general-math}
\usepackage{packages-styles/field-probabilitytheory}
\usepackage{packages-styles/field-algorithmic}
\usepackage{packages-styles/general-scientific}
\newcommand\Language{english}
\usepackage{packages-styles/general-text}

\usepackage{xcolor}
\usepackage{arxiv}

\hypersetup{
pdftitle={A Proof of Kirchhoff's First Law for Hyperbolic Conservation Laws on Networks},
pdfsubject={Network Science},
pdfauthor={},
pdfkeywords={Network Science, Analysis of PDEs, Dynamical Systems, Mathematical Modeling}
}

\newcommand{\NN}{\mathcal{N}}
\Crefname{property}{Property}{Properties}
\setlist[description]{font=\rmfamily}

\title{A Proof of Kirchhoff's First Law for Hyperbolic Conservation Laws on Networks}

\author{
 Alexandre M. Bayen \\
  Department of Electrical Engineering and Computer Sciences\\
  Department of Civil and Environmental Engineering\\
  University of California, Berkeley\\
  United States of America\\
  \href{bayen@berkeley.edu}{\texttt{bayen@berkeley.edu}}\\
   \And
 Alexander Keimer\\
  Department of Mathematics\\
  Friedrich-Alexander-University Erlangen-N{\"u}rnberg (FAU)\\
  Germany\\
  \href{mailto:alexander.keimer@fau.de}{\texttt{alexander.keimer@fau.de}}\\
  \And
 Nils Müller\footnotemark\\
  Max Planck Institute for Software Systems\\
  Saarland Informatics Campus\\
  Germany\\
  \href{mail@nilsmueller.io}{\texttt{mail@nilsmueller.io}}\\
}

\begin{document}
\maketitle

\begin{abstract}
    Networks are essential models in many applications such as information technology, chemistry, power systems, transportation, neuroscience, and social sciences. In light of such broad applicability, a general theory of dynamical systems on networks may capture shared concepts, and provide a setting for deriving abstract properties.\\
    To this end, we develop a calculus for networks modeled as abstract metric spaces and derive an analog of Kirchhoff's first law for hyperbolic conservation laws.
    In dynamical systems on networks, Kirchhoff's first law connects the study of abstract global objects, and that of a computationally-beneficial edgewise-Euclidean perspective by stating its equivalence. In particular, our results show that hyperbolic conservation laws on networks can be stated without explicit Kirchhoff-type boundary conditions.
\end{abstract}


\input{writing/introduction}

\input{writing/theory.tex}
\newpage

\printbibliography
\end{document}

%% file: writing/introduction.tex
\section{Introduction}

In the dynamical systems community, networks are largely studied as families of Euclidean edges $\mathcal{E} := ([0, w_e])_{e \in E}$, with $w_e \in \R_{>0}$ and $E$ being an index set. The network's structure is often only incorporated by defining a family of coupled dynamical systems $\Phi_e: [0, T] \times \mathcal{E} \to \mathcal{S}, e \in E$, where $\mathcal{S}$ models some state-space and $T \in \R_{> 0}$. Networks typically do not appear explicitly.\\
In this work, however, we construct networks as metric spaces $(\mathcal{N}, d)$, where $d$ is a notion of path distance. We will show that sufficiently \emph{regular} networks $(\mathcal{N}, d)$ are Polish metric spaces---a convenient setting for deriving basic measure-theoretic notions underpinning abstract analysis. Although Euclidean calculus does not immediately apply to such networks, we develop, with minimal assumptions, a metric calculus for networks that turns out to have many of the properties of the known Euclidean setting.\\
Using such network calculus we introduce a weak notion of hyperbolic conservation laws, for which our main theorem, \cref{thm:kirchhoff}, states equivalence with $(\Phi_e)_{e \in E}$ if all $\Phi_e$ are a solution to a hyperbolic conservation law on an Euclidean edge with coupled boundary conditions. This holds, in particular, if all $\Phi_e$ admit a quasi-linear form (see \cref{rem:kirchhoff}). \cref{thm:kirchhoff} therefore implies that an arguably universal calculus for networks exists for which hyperbolic conservation laws can be stated without explicit Kirchhoff-type boundary conditions.\\
The results of our work thereby contribute to unifying the study of global objects relevant to dynamical systems, such as time-limiting phenomena or approximations, with the study of physically-plausible transport models in a mostly Euclidean setting.

\paragraph{Motivation and Related Work.}

Rather recently, the study of networks was extended by a global---as opposed to edgewise-Euclidean---perspective based on modeling networks as abstract metric spaces \cite{kuchment2003,Mugnolo_2014,duef2022}. While such network spaces, often also called quantum graphs, have been studied widely before \cite{pauling1936,platt1949,richardson1972,kottos1999,mugnolo2007,berkolaiko2022}, an abstract global model seeks to describe associated phenomena and apply existing mathematical language that benefits from it.\\ Physically-plausible transport on networks, modeled by \emph{coupled Euclidean hyperbolic conservation laws}, has been of great interest to the mathematical, science, and engineering communities \cite{holden1995,coclite2005,gugat2005,bretti2008,bressan2014,guarguaglini2015,gugat2015,garavello2016,bayen2019,bayen2022}. We seek to derive such \emph{coupling} in hyperbolic conservation laws as a consequence of the network's geometry and an associated abstract calculus characterized by universal properties. With such derivation, Kirchhoff's first law for hyperbolic conservation laws is provable and replaces a set of model axioms.

\begin{figure}
\vspace{-4mm}
\centering
    \begin{minipage}{.5\textwidth}
      \centering
      \includegraphics[width=1.\linewidth]{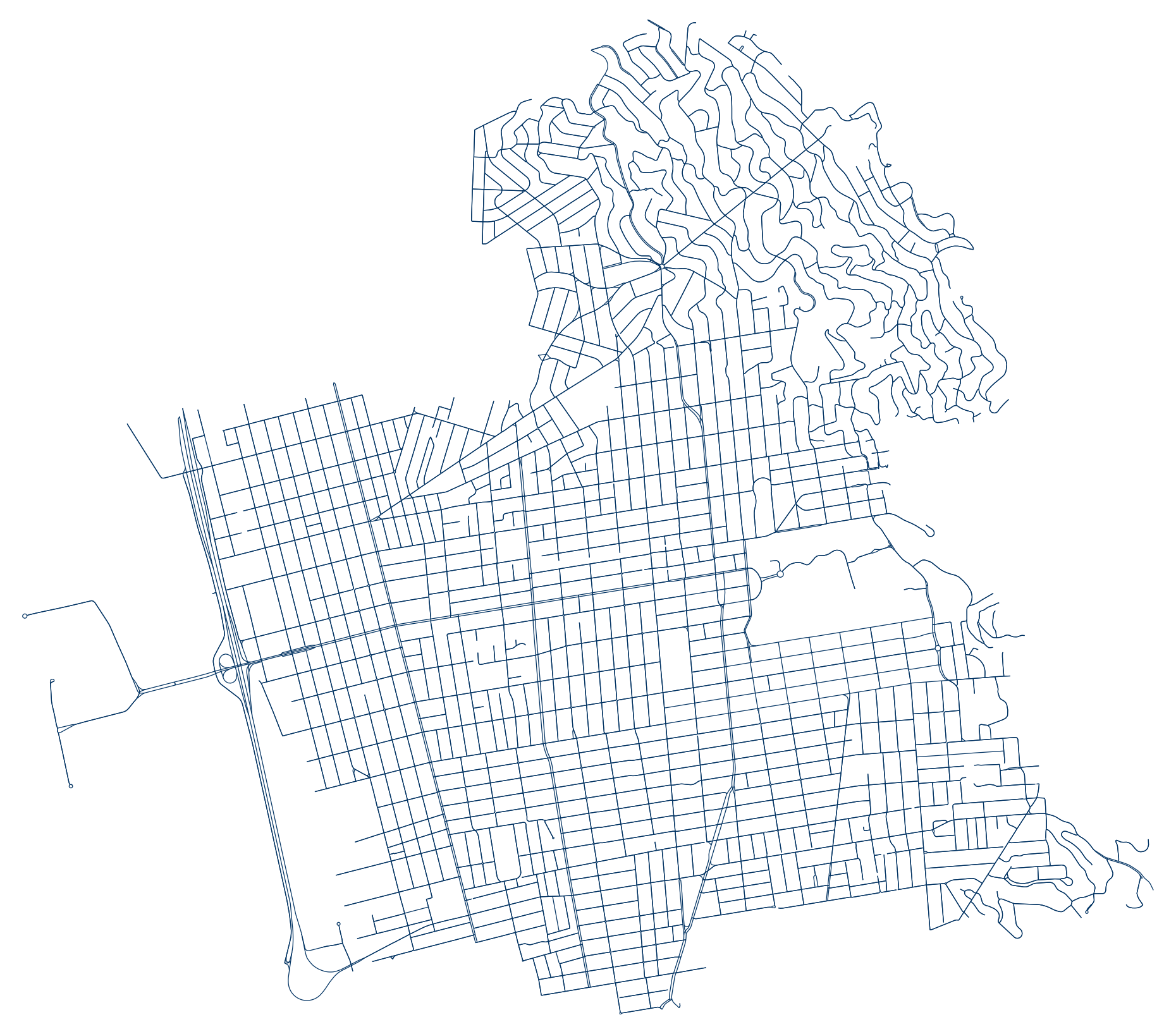}
    \end{minipage}%
    \begin{minipage}{.5\textwidth}
      \centering
      \includegraphics[width=1.\linewidth]{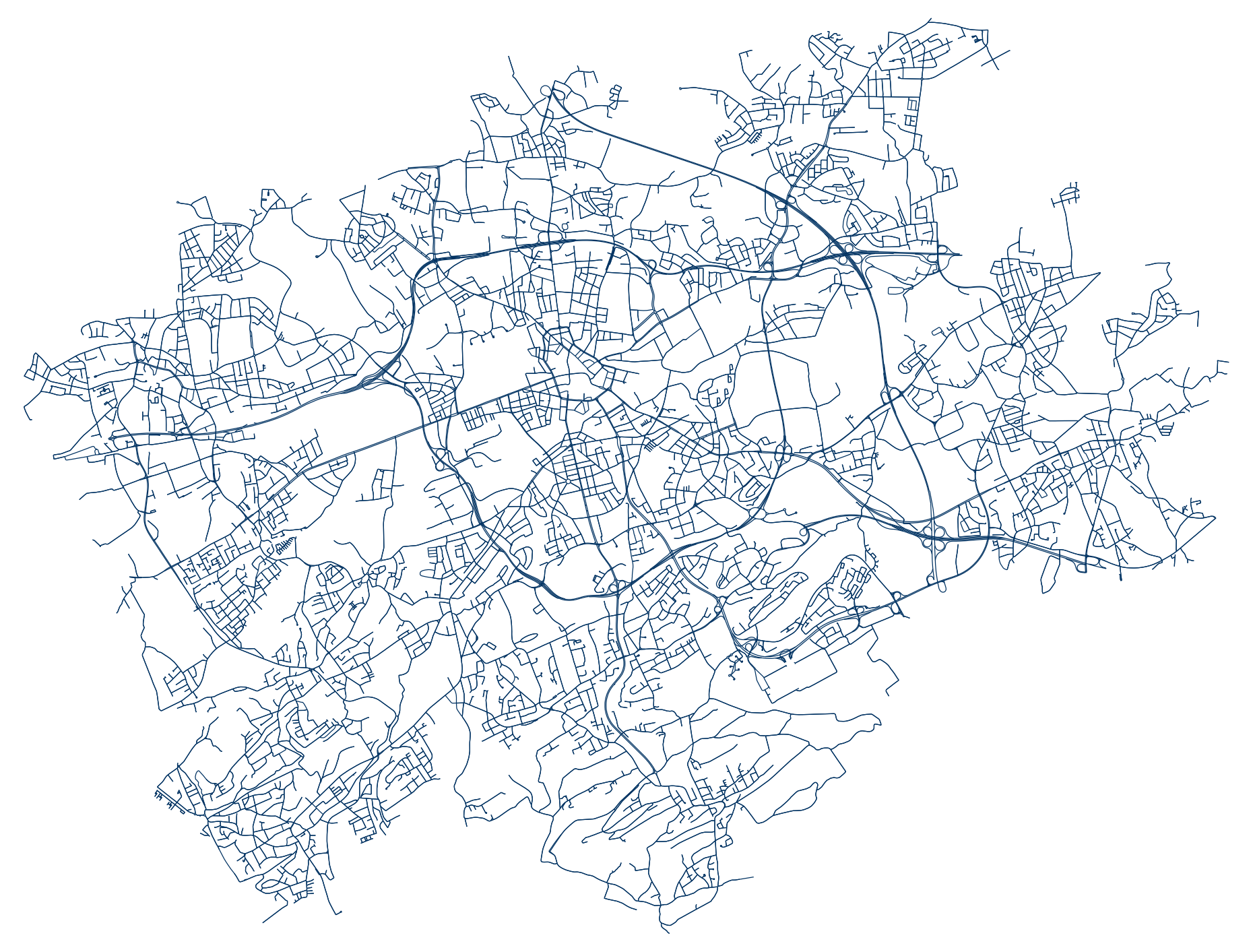}
    \end{minipage}
    \caption{The networks that represent the drivable roads of Berkeley, CA, USA (left) and Bochum, Germany (right) projected onto $\R^2$. Created from OpenStreetMap data \cite{OpenStreetMap} using the tool OSMnx \cite{boeing2017}.}
\label{fig:bbdriveembed}
\end{figure}

\paragraph{Outline.} In \cref{sec:networks}, we introduce a particularly accessible setting for constructing networks based on graph data and equivalence relations, discuss the (un)suitability of graphs for the purpose of network modeling, rule out pathological examples of networks to develop a notion of regularity, and motivate the notion of path distance.\\
We continue constructing a Lebesgue-like measure for networks, developing differentiability, derivatives, and integration by parts in \cref{sec:calculus}.\\
We close with motivating a notion of a weak hyperbolic conservation law for networks in \cref{sec:conservation-laws}, which lets us derive our main theorem, \cref{thm:kirchhoff}, an analog of Kirchhoff's law for these abstract dynamical systems.

%% file: writing/theory.tex
\section{Networks}
\label{sec:networks}

The initial goal of this work is to introduce a rigorous notion of a \emph{network}. Networks ought to be the central objects on which the rest of the theory will build. A guiding perspective is taken in classical mechanics, built into many of the mathematical terms used in this section. The decisive impact of this perspective will be demonstrated first.\\

Due to the widespread use of \emph{graphs} in discrete mathematical modeling of related real-world phenomena, it is of interest to discuss their relevance to the purpose of network modeling. Classically, only when modeling pairwise relations of elements of a set $V$ called \emph{vertices}, one resorts to graphs $(V, E)$, with $E: K \to V \times V$, where $E(k)$ is called an $\emph{edge}$, $K$ is an index set, and $k \in K$. This representation includes multiple edges with the same vertices. Possibly one assigns a weight $W: K \to \R_{>0}$ to each edge to describe additional structure.

\begin{definition}[Positively weighted graph]
    A \textbf{positively weighted graph} is a tuple $(V, E, W)$, where $V$ is a set, $E: K \to V \times V$ is a function from an index set $K$ to the product of $V$ with itself, and $W: K \to \R_{>0}$ is a function from $K$ to the positive reals. It is required that $E(K)_1 \cup E(K)_2 = V$, that is, all vertices are connected to an edge.
\end{definition}

As a network model and for the purpose of modeling in this work, however, graphs will not be sufficient. In mathematical terms, one can find an indication of the unsuitability of graphs in \cref{lem:nographs}.

\begin{lemma}
    Every metric, path-connected, finite vertex space has only one vertex.
\label{lem:nographs}
\end{lemma}
\begin{proof}
    A finite metric space $(V, d)$ has the discrete topology as for all $x \in V$ the ball with radius smaller than
    \[
        \min_{y \in V\setminus\{x\}} d(x,y)
    \]
    contains only $x$. If $V$ carries the discrete topology all subsets of $V$ are \emph{clopen}. If $V$ is additionally path-connected, it is also connected, and thus the only clopen sets are $\varnothing$ and $V$. This implies $V = \{x\}$.
\end{proof}

\cref{lem:nographs} characterizes the incompatibility of graphs and the \emph{continuity} of real numbers embedded into much of the mathematical language. This is relevant to practical problems, as continuity is the defining structure in \emph{continuum mechanics} and is thus an inherent property in modeling network phenomena. In applications, the following, slightly informally stated properties are expected to be fulfilled.
\begin{align}
    &\text{\emph{A network ought to be path-connected, admit a compatible notion of distance,}}\notag\\
    &\text{\emph{be complete, and have reasonable integration.}}
    \label[property]{prop:structural}
\end{align}

Nonetheless, it turns out that one can construct networks that fulfill \cref{prop:structural} \textbf{from} graphs but \textbf{not as} graphs---as will be shown by \cref{def:network} and \cref{thm:regularity}. That is, positively weighted graphs non-uniquely encode the information of a specific network among all networks. Yet additional information is supplied by intervals of real numbers to obtain a general network setting with \cref{prop:structural}.

\subsection{Construction of Networks}

A quite elementary way to construct a network with such properties is as a set of equivalence classes of elements of intervals in $\R$ with a notion of path-distance.
While there are multiple ways to construct such spaces mathematically, even such that do not rely on graphs, in \cref{def:network} networks will be constructed as a set and equipped with a distance function.
Some pathological examples of networks will then be ruled out in \cref{def:regularnetwork}, and \cref{prop:structural} proven in \cref{thm:regularity}.
This presentation may be relatively accessible for interdisciplinary study, as it requires only a small amount of abstract language. Two examples of street networks are visualized in \cref{fig:bbdriveembed}.\\

A topological construction of networks, on the other hand, is given in \cite{mugnolo2019}. Although the topological setting is arguably elegant, the presented quotient topology in general agrees with the path distance topology only for (in the work termed \enquote{combinatorically}) locally finite networks (see \enquote{middle} of \cref{fig:malignnetworks}). As a benefit of such a setting, it becomes clear that networks are topologically characterized by the universal property of the quotient topology. In this work, it will be shown that the assumption of a countable number of edges can be recovered from the regularity assumptions of \cref{def:regularnetwork}.\\

To define networks in a concise way, the following tools are introduced. Given a family $(I(w_j))_{j \in J}$ of intervals $I(w_j) := [0, w_j] \subseteq \R$ with $w_j>0$ for all $j \in J$ and $J$ being a non-empty index set.
The disjoint union of these intervals is defined as
\[
     \mathrm{Un}\big((I(w_j))_{j \in J}\big) := \{ (x, j) \mid x \in I(w_j)\,,\,\, j \in J \} \,,
\]
and an extended metric on $\mathrm{Un}\big((I(w_j))_{j \in J}\big)$ as
\[
\begin{array}{c}
    d': \mathrm{Un}\big((I(w_j))_{j \in J}\big) \times \mathrm{Un}\big((I(w_j))_{j \in J}\big) \longrightarrow [0, \infty]
    \,,
    \\
    \big((x, j), (x', j')\big) \longmapsto  \begin{cases} \abs{x-x'}, & \text{if $j=j'$}  \\ \infty, & \text{otherwise.}\end{cases}
\end{array}
\]

\begin{definition}[Network]\leavevmode
\label{def:network}
    Given a positively weighted graph $(V, E, W)$, one can define an equivalence relation $\sim$ on the disjoint union $\mathrm{Un}\big((I(W(k)))_{k \in K}\big)$ by
    \[
        (x, k) \sim (x', k') \longeq 
        \begin{cases}
            x=x' \,\, \text{and} \,\, E(k)=E(k') & \text{or}
            \\
            x=x'=0 \,\, \text{and} \,\, E(k)_1=E(k')_1 & \text{or}
            \\
            x=0 \,\, \text{and} \,\, x'=W(k') \,\, \text{and} \,\, E(k)_1=E(k')_2 & \text{or}
            \\
            x=W(k) \,\, \text{and} \,\, x'=0 \,\, \text{and} \,\, E(k)_2=E(k')_1 & \text{or}
            \\
            x=W(k) \,\, \text{and} \,\, x'=W(k') \,\, \text{and} \,\, E(k)_2=E(k')_2 \,.
        \end{cases}
    \]
    A \textbf{network} associated with $(V, E, W)$ is defined as the set $\mathcal{N}$ of equivalence classes
    \[
        \mathcal{N} := \left\{ [x] \mid x \in \mathrm{Un}\big((I(W(k)))_{k \in K}\big) \right\} \,,
    \]
    where
    \[
        [x] := \{ y \in \mathrm{Un}\big((I(W(k)))_{k \in K}\big) \mid  \,\, y \sim x \} \,.
    \]
    A network is endowed with a function
    \[
    \begin{array}{c}
        d: \mathcal{N} \times \mathcal{N} \longrightarrow [0, \infty]
        \,,
        \\
        (x, x') \longmapsto \inf\left\{ \displaystyle\sum_{i=1}^n d'(p_i, q_i) \,\,\bigg\vert\,\, \begin{array}{c}
            n \in \N \,;\,\, (p_i)_{i \in \N_{\leq n}}, (q_i)_{i \in \N_{\leq n}} \in \mathrm{Un}\big((I(W(k)))_{k \in K}\big)
            \\
            x \sim p_1\,;\,\, x' \sim q_n \,;\,\, \forall i \in \N_{[1, n-1]}: 
            q_i \sim p_{i+1}
        \end{array} \right\} \,.
    \end{array}
    \]
\end{definition}

\begin{remark}
    In $\mathcal{N}$, the equivalence classes of the boundaries of intervals are called \textbf{vertices} of the network and denoted with $\mathcal{V}$. The set $\mathcal{E}$ is called \textbf{edges} of the network $\NN$, where an element is a set that contains the equivalence classes of elements with the same index. The vertices of an edge $e \in \mathcal{E}$ are denoted by $0_e$ and $w_e$. Note that the distance function $d$ does not regard the orientation of the edges. The edges that contain the vertex $v$ will be denoted with $\mathcal{E}_v$.
\end{remark}

\begin{figure}[ht]
    \color{myBlue}
    \centering
    \makebox[\textwidth][c]{
    \begin{tikzpicture}[scale=2.5]
            \draw[very thick]  (-1.866,0) -- node[above=3mm] {$E_W(1)$} (-.866,0);
            \draw[very thick]  (-.866,0) -- node[above=3mm] {$E_W(2)$} (0,.5);
            \draw[very thick]  (-.866,0) -- node[below=3mm] {$E_W(3)$} (0,-.5);
            \draw[very thick]  (0,.5) -- node[above=3mm] {$E_W(5)$} (.866,0);
            \draw[very thick]  (0,-.5) -- node[below=3mm] {$E_W(6)$} (.866,0);
            \draw[very thick]  (.866,0) -- node[above=3mm] {$E_W(7)$} (1.866,0);
            \draw[very thick]  (0,.5) -- node[right=1mm] {$E_W(4)$} (0,-.5);
            
            \node[label={[label distance=-11mm]:{\Large $1$}}] at (-1.866,0) [circle,fill=myBlue,scale=.65] {};
            \node[label={[label distance=-11mm]:{\Large $2$}}] at (-.866,0)
            [circle,fill=myBlue,scale=.65] {};
            \node[label={[label distance=-11mm]:{\Large $6$}}] at (1.866,0) [circle,fill=myBlue,scale=.65] {};
            \node[label={[label distance=-11mm]:{\Large $5$}}] at (.866,0) [circle,fill=myBlue,scale=.65] {};
            \node[label={[label distance=3mm]:{\Large $3$}}] at (0,.5) [circle,fill=myBlue,scale=.65] {};
            \node[label={[label distance=-11mm]:{\Large $4$}}] at (0,-.5) [circle,fill=myBlue,scale=.65] {};
    \end{tikzpicture}
    }
    \caption{An embedding of the Wheatstone network into $\R^2$ that locally preserves vertex distances.}
    \label{fig:wheatstone}
\end{figure}
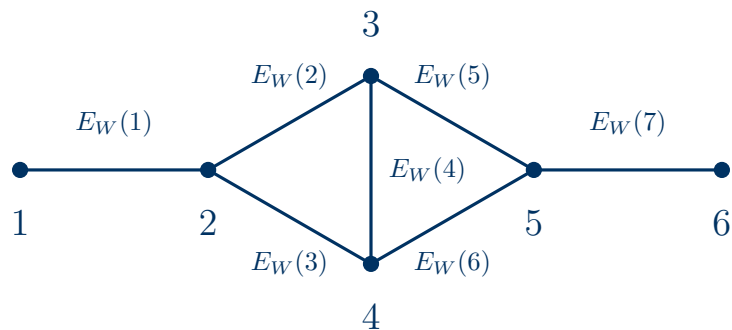

Essentially, intervals are \enquote{glued} together, and a universal extension from distances along intervals to distances along simple paths is constructed. Although at this point, paths have not been defined rigorously. In fact, paths will be recovered from the distance function $d$\,---which will be shown to be a metric---in the case of regular networks.\\

A classic example of a network is the Wheatstone network, which originated from electrodynamics \cite{wheatstone1843} and is seminal in the study of transportation systems \cite{braess1968}.

\begin{example}[Wheatstone network]
    As per \cref{def:network}, one can characterize a network by a positively weighted graph $(V_{wh}, E_{wh}, W_{wh})$, where $V_{wh}:=\{1, \dots, 6\}$, $K_{wh} := \{1, \dots, 7\}$, and $E_{wh}: K_{wh} \to V_{wh} \times V_{wh}$, and
    \begin{gather*}
        E_{wh}(1) := (1, 2) \,,
        \quad
        E_{wh}(2) := (2, 3) \,,
        \quad
        E_{wh}(3) := (2, 4) \,,
        \quad
        E_{wh}(4) := (3, 4) \,,
        \\
        E_{wh}(5) := (3, 5) \,,
        \quad
        E_{wh}(6) := (4, 5) \,\text{, and}
        \quad
        E_{wh}(7) := (5, 6) \,.
    \end{gather*}
    The weight function can be picked rather arbitrarily, but an embedding into $\R^2$ preserves vertex distances, e.g., for $W_{wh} \equiv 1$. The network $\NN_{wh}$ characterized by the graph $(V_{wh}, E_{wh}, W_{wh})$ is called \textbf{Wheatstone network}. See \cref{fig:wheatstone} for a visualization of the Wheatstone network.
\label{ex:wheatstonenetwork}
\end{example}

\subsection{Properties of Networks}

There are certain networks, depicted in \cref{fig:malignnetworks}, that do not carry the structure required in \cref{prop:structural}. These networks turn out to be somewhat pathological and do not lend themselves well to the purpose of modeling real-world networks. However, the following definition describes networks that fulfill all of the requirements of \cref{prop:structural}.

\begin{definition}[Regularity]
    A network $\NN$ is called \textbf{regular} if
    \begin{itemize}
        \item it is locally finite, i.e., each vertex is included in at most a finite number of edges,
        
        \item there exists a positive lower bound of the edge lengths, and
        \item it is connected with respect to the topology induced by $d$.
    \end{itemize}
\label{def:regularnetwork}
\end{definition}

Regular networks allow the description of Lebesgue integration, continuous functions, and paths. These constructs are essential to the theory and practice of modeling transport phenomena on networks. The following theorem, therefore, describes several mathematical qualities of networks that enable their study.

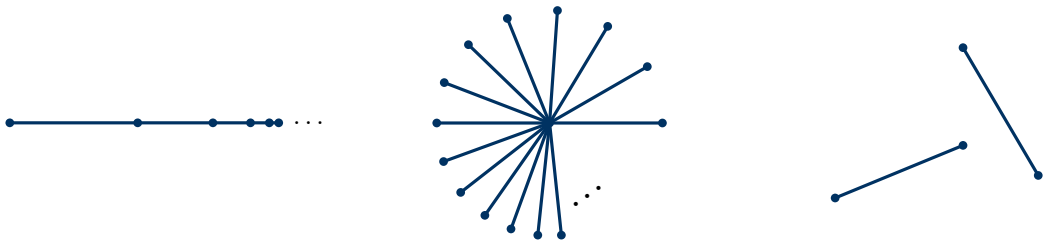
\begin{figure}[ht]
\centering
    \begin{minipage}{.33\textwidth}
    \centering
    \begin{tikzpicture}
        \draw[very thick,color=myBlue] (0,0) -- (1.7,0) ;
        \draw[very thick,color=myBlue] (1.7,0) -- (2.7,0);
        \draw[very thick,color=myBlue] (2.7,0) -- (3.2,0);
        \draw[very thick,color=myBlue] (3.2,0) -- (3.45,0);
        \draw[very thick,color=myBlue] (3.45,0) -- (3.575,0);
        \node at (0,0) [circle,fill=myBlue,scale=.35] {};
        \node at (1.7,0) [circle,fill=myBlue,scale=.35] {};
        \node at (2.7,0) [circle,fill=myBlue,scale=.35] {};
        \node at (3.2,0) [circle,fill=myBlue,scale=.35] {};
        \node at (3.45,0) [circle,fill=myBlue,scale=.35] {};
        \node at (3.575,0) [circle,fill=myBlue,scale=.35] {};
        \node at (4.0,0) {$\dots$};
    \end{tikzpicture}
    \end{minipage}%
    \begin{minipage}{.33\textwidth}
    \centering
    \begin{tikzpicture}
        \draw[very thick,color=myBlue] (0,0) -- (0:1.5cm);
        \draw[very thick,color=myBlue] (0,0) -- (30:1.5cm);
        \draw[very thick,color=myBlue] (0,0) -- (59:1.5cm);
        \draw[very thick,color=myBlue] (0,0) -- (86:1.5cm);
        \draw[very thick,color=myBlue] (0,0) -- (112:1.5cm);
        \draw[very thick,color=myBlue] (0,0) -- (136:1.5cm);
        \draw[very thick,color=myBlue] (0,0) -- (159:1.5cm);
        \draw[very thick,color=myBlue] (0,0) -- (180:1.5cm);
        \draw[very thick,color=myBlue] (0,0) -- (200:1.5cm);
        \draw[very thick,color=myBlue] (0,0) -- (218:1.5cm);
        \draw[very thick,color=myBlue] (0,0) -- (235:1.5cm);
        \draw[very thick,color=myBlue] (0,0) -- (250:1.5cm);
        \draw[very thick,color=myBlue] (0,0) -- (264:1.5cm);
        \draw[very thick,color=myBlue] (0,0) -- (276:1.5cm);
        \node at (0,0) [circle,fill=myBlue,scale=.35] {};
        \node at (0:1.5cm) [circle,fill=myBlue,scale=.35] {};
        \node at (30:1.5cm) [circle,fill=myBlue,scale=.35] {};
        \node at (59:1.5cm) [circle,fill=myBlue,scale=.35] {};
        \node at (86:1.5cm) [circle,fill=myBlue,scale=.35] {};
        \node at (112:1.5cm) [circle,fill=myBlue,scale=.35] {};
        \node at (136:1.5cm) [circle,fill=myBlue,scale=.35] {};
        \node at (159:1.5cm) [circle,fill=myBlue,scale=.35] {};
        \node at (180:1.5cm) [circle,fill=myBlue,scale=.35] {};
        \node at (200:1.5cm) [circle,fill=myBlue,scale=.35] {};
        \node at (218:1.5cm) [circle,fill=myBlue,scale=.35] {};
        \node at (235:1.5cm) [circle,fill=myBlue,scale=.35] {};
        \node at (250:1.5cm) [circle,fill=myBlue,scale=.35] {};
        \node at (264:1.5cm) [circle,fill=myBlue,scale=.35] {};
        \node at (276:1.5cm) [circle,fill=myBlue,scale=.35] {};
        \draw (300:1.cm) node {\textbf{\reflectbox{$\ddots$}}};
    \end{tikzpicture}
    \end{minipage}
    \begin{minipage}{.33\textwidth}
    \centering
    \begin{tikzpicture}
        \draw[very thick,color=myBlue] (0,0) -- (1.7,.7) ;
        \draw[very thick,color=myBlue] (1.7,2) -- (2.7,.3);
        \node at (0,0) [circle,fill=myBlue,scale=.35] {};
        \node at (1.7,.7) [circle,fill=myBlue,scale=.35] {};
        \node at (1.7,2) [circle,fill=myBlue,scale=.35] {};
        \node at (2.7,.3) [circle,fill=myBlue,scale=.35] {};
    \end{tikzpicture}
    \end{minipage}
    \caption{Examples motivating \cref{def:regularnetwork}: Even a locally finite, connected network \textbf{may not be complete} (left). Even if edge lengths have a lower bound, a network \textbf{may not be locally compact} (middle). A network which is \textbf{not connected} (right).}
    \label{fig:malignnetworks}
\end{figure}

\begin{theorem}[Regularity]
    Regular networks $(\mathcal{N}, d)$ are path-connected, complete, locally compact, separable metric spaces with a countable number of edges.
\label{thm:regularity}
\end{theorem}
\begin{remark}
    Regularity therefore implies that networks are Polish spaces. \textbf{All networks considered will be regular for the remainder of this work}.
\end{remark}
\begin{proof}\leavevmode
    \begin{enumerate}
        \item It is shown that $d$ is an extended metric (may take the value $\infty$) on a regular network.
        \setlist[description]{font=\normalfont\itshape}
        \begin{description}
            \item[Identitity of indiscernibles:] Let $x,y \in \NN$. If $x \neq y$ their distance $d(x,y)$ is not smaller than the length of a subinterval of some edge. To see why that is, consider the following cases.
            \begin{description}
            \item[$x$ is a vertex:] Either $y$ must be an inner point of one of the finite number of edges connected to $x$. In this case (in slight abuse of notation) $d(x,y) \geq \min \{\abs{x-y}, \varepsilon\}$, where $\varepsilon$ is a positive lower bound to the edge lengths. Otherwise, $y$ is contained in an edge not connected to $x$. In this case $d(x,y) \geq \varepsilon$.
            \item[$x$ is not a vertex:] $d(x,y)$ is larger or equal to the shortest distance of $x$ to the closest vertex (see previous case) or (in slight abuse of notation) $d(x,y) = \abs{x-y}$ if $x$ and $y$ are on the same edge.
            \end{description}
            If $x=y$, one has $d(x,y) = 0$ as for $p_1 \sim x$, $q_1 \sim y$ there is the extended metric distance $d'(p_1, q_1) = 0$.
            
            \item[Symmetry:] By the reversal of finite sequences and commutativity of addition, $d$ is symmetric.
            
            \item[Triangle inequality:] Let $x,y,z \in \NN$ and for $\gamma > 0$ pick sequences $(p_i, q_i)_{i \in \mathbb{N}}, (p_i', q_i')_{i \in \mathbb{N}}$ (see \cref{def:network}) with
            \begin{gather*}
                \left(\sum_{i=1}^n d'(p_i, q_i)\right) - d(x,y) < \tfrac{\gamma}{2} \qquad \text{and} \qquad \left(\sum_{i=1}^n d'(p_i', q_i') \right) - d(y,z) < \tfrac{\gamma}{2}  \,,
                \\
                x \sim p_1, y \sim q_n, y \sim p_1', z \sim q_n' \,, \,\, \text{and}
                \\
                \forall i \in \N_{<n}: q_i \sim p_{i+1}, q_i' \sim p_{i+1}' \,.
            \end{gather*}
            Such sequences exist by definition of $d$.
            By concatenation of the sequences, one has
            \[
                d(x, z) - d(x,y) - d(y, z) < \gamma \,.
            \]
            As this holds for all $\gamma > 0$, the triangle inequality follows.
        \end{description}
        
        \item It is shown that regular networks $(\NN, d)$ are path-connected. By assumption, regular networks are connected. Regular networks are also locally path-connected by the argument following below. Generally, topological spaces that are connected and locally path-connected are path-connected. All of the above arguments are with respect to the topology induced by the extended metric $d$.\\
        
        Let $x \in \NN$, then one can pick $\varepsilon>0$ such that $U_\varepsilon(x)$ contains only elements of $\NN$ that share an edge with $x$. This is done such that $U_\varepsilon(x)$ is only a subset of one edge if $x$ is not a vertex. Otherwise, $\varepsilon$ should be picked as a lower bound of the edge lengths. $U_\varepsilon(x)$ is path-connected as edges are path-connected, and its restriction onto each edge is a subinterval.
        
        \item By the following argument, one can show using the path-connectedness of $\NN$ that the extended metric $d$ takes only finite values, i.e., that $d$ is a metric.\\
        
        For $x,y \in \NN$ pick a path $\gamma: [0, 1] \to \NN$ with $\gamma(0) = x$ and $\gamma(1) = y$. As $[0, 1]$ is compact, $\gamma([0, 1])$ is also compact and in particular bounded. Therefore, $d(x,y)$ is bounded by the same bound.
        
        \item A regular network $(\NN, d)$ has a countable number of edges.\\
        
        Let $\varepsilon>0$ be positive lower bound of the edge lengths and $x \in \mathcal{V}$. The initial goal is to show that for all $n \in \N_0$
        \[
            A_{\varepsilon n} := \{ y \in \mathcal{V} \mid d(x,y) \leq \varepsilon n\}
        \]
        is finite. This can be proved by an inductive argument. Clearly, $\abs{A_\varepsilon} = 0$ and thus assume that $A_{\varepsilon n}$ is finite for some $n \in \N_0$. The vertices in $A_{\varepsilon(n+1)} \setminus A_{\varepsilon n}$ are connected to $A_{\varepsilon n}$ by single edges, as per the lower bound of edge lengths $\varepsilon$. But as $A_{\varepsilon n}$ is finite and all $v \in A_{\varepsilon n}$ are only connected to a finite number of edges, one has
        \[
            \abs{A_{\varepsilon (n+1)}} \leq \sum_{v \in A_{\varepsilon n}} \abs{\mathcal{E}_v} < \infty \,.
        \]
        This implies that by countable union of finite sets, the set of vertices $\bigcup_{n \in \N_0} A_{\varepsilon n}$ is finite, and the set of edges is countable.
        
        \item A regular network $(\NN, d)$ is a complete space, as for all Cauchy sequences $(x_n)_{n \in \N} \in \NN$ one can pick $N \in \N$ such that $(x_{n+N})_{n \in \N}$ is contained in a finite number of edges, which by restriction, is therefore included in a compact and thus complete subset.
        
        \item A regular network $(\NN, d)$ is locally compact, as it is complete and neighborhoods contain bounded closed neighborhoods.
        
        \item Generally countable unions of separable spaces are separable. Thus, regular networks are separable as they are countable unions of edges, which are separable.
    \end{enumerate}
\end{proof}

Lastly, for the following work, a precise notion of a path and its length is useful as it relates $d$ and shortest paths.

\begin{definition}[Path]
    A \textbf{path} in a regular network $\mathcal{N}$ is a continuous function $p: [0, l] \to \mathcal{N}$, where $l > 0$. The \textbf{length of a path} is defined as
    \[
        \ell(p) := \sup \left\{ \sum_{i=1}^{n-1} d(p(t_i), p(t_{i+1})) \,\,\bigg\vert\,\, 0=t_1 \leq \dots \leq t_n=l,\, n \in \N \right\} \,.
    \]
    A path $p: [0, l] \to \mathcal{N}$ is called \textbf{parameterized by its length} if $\ell(p_{|[0, l']}) = l'$ for all $l'\leq l$.
\end{definition}

\begin{remark}
    Given an edge $e_k \in \mathcal{E}$, in slight abuse of notation $\ell(e_k):=W(k)$ will also be written for any $k \in K$, where $K$ is the index set of the graph associated with the network. For a path $p:[0, l] \to \NN$, if $p(0) = x$ and $p(l)=y$, it is said that $x \rightsquigarrow y$ for $p$.
\end{remark}

\begin{theorem}
    In a regular network $\mathcal{N}$, for any $x,y \in \mathcal{N}$, there exists at least one shortest path $x \rightsquigarrow y$ parametrized by its length $d(x,y)$.
\end{theorem}
\begin{proof}
    For $x,y \in \NN$, set $\varepsilon := d(x,y)$. The lower bound on edge lengths implies that there is a finite number of paths $x \rightsquigarrow y$ in $\NN \cap U_\varepsilon(x)$ that are parametrized by the length of each edge they can be restricted to. One of such that is shortest exists in this finite set of paths and must have length $d(x,y)$ by definition of $d$.
\end{proof}

\section{Calculus on Networks}
\label{sec:calculus}

One may desire a \emph{medium} to localize a notion of mass in networks and describe concentrations thereof. This mass may be a homogeneous object, unlike, e.g., a set of distinguished agents. This perspective is particularly meaningful in what one may call \enquote{egalitarian systems}, as those of infrastructure or platform providers. Mathematical measure theory provides a framework to study such objects and includes a notion of signed, discrete, and continuous media.

\subsection{Measures and Integration}

There is a unique measure on a network that assigns each edge and subinterval thereof its length. Measures of this type are essential for classical integration and the modeling of densities. As regular networks are metric spaces, one can consider their Borel $\sigma$-algebra $\mathcal{B}(\NN)$ as a set of measurable sets. An example of a discrete measure, often encountered in practice in the form of data, is pictured in \cref{fig:discretemeasure}.

\begin{figure}[ht]
    \centering
    \includegraphics[width=9cm]{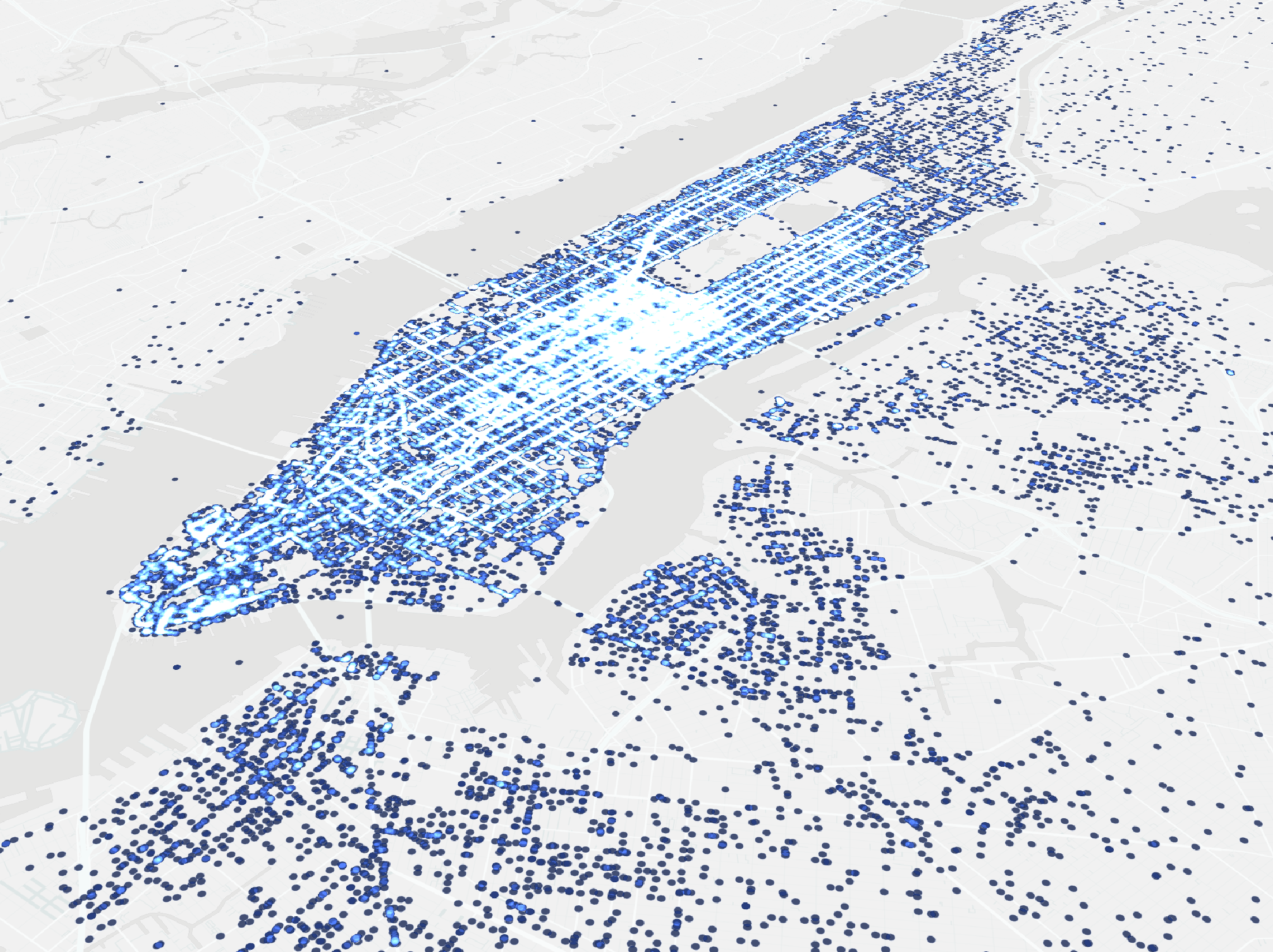}
    \caption{The discrete measure representing New York City taxi drop-offs on January 15, 2015. Data taken from NYC Taxi and Limousine Commission. Inspired by and created with kepler.gl.}
    \label{fig:discretemeasure}
\end{figure}

\begin{definition}[Lebesgue measure]
    A measure $\lambda$ is called the \textbf{Lebesgue measure} on a regular network $(\mathcal{N}, \mathcal{B}(\mathcal{N}))$ if for all $U \subseteq \NN$ that are isometrically isomorphic to a possibly degenerate interval, i.e. $U \cong [0, l]$, $\lambda_{|U}$ is the classical Lebesgue measure.
\label{def:lebesgue}
\end{definition}
\begin{remark}
    Two metric spaces $(U, d_U), (V, d_V)$ are \textbf{isometrically isomorphic} if there exists a bijective function $f: U \to V$that preserves the metric distance, i.e., for all $u_1, u_2 \in U$
    \[
        d_U(u_1, u_2) = d_V(f(u_1), f(u_2)) \,.
    \]
    One writes $U \cong V$.
\end{remark}
\begin{lemma}
    On regular networks, a unique Lebesgue measure exists and is $\sigma$-finite.
\label{lem:lebesgue}
\end{lemma}
\begin{proof}
    For all elements $x \in \NN$, one finds a neighborhood $U_x$ of $x$ that is isometrically isomorphic to a possibly degenerate interval. Due to the compactness of single edges and the countability of the edges (see \cref{thm:regularity}), one can find a countable set of elements $B \subset \NN$ such that
    \[
        \bigcup_{x \in B} U_x = \NN \,.
    \]
    Without loss of generality, assume that $(U_x)_{x \in B}$ are disjoint, as otherwise a disjoint cover can be constructed from  $(U_x)_{x \in B}$ by subtraction and intersection of sets.
    Therefore, by $\sigma$-additivity every such measure $\lambda$ must for all $A \in \mathcal{B}(\NN)$ satisfy the equation
    \[
        \lambda(A) = \lambda\bigg(\bigcup_{x \in B} A \cap U_x \bigg) = \sum_{x \in B} \lambda(A\cap U_x) = \sum_{x \in B} \lambda_{|U_x}(A\cap U_x) \,,
    \]
    where $\lambda_{|U_x}$ is the classical Lebesgue measure on the interval $U_x$. The right hand serves as the unique definition of a networks's Lebesgue measure, as for other covers of $\NN$ one can show equality by refinement of the cover $(U_x)_{x \in B}$. The Lebesgue measure is clearly $\sigma$-finite on regular networks as $U_x$ includes only a vertex and has measure zero, or is included in an edge with a finite length.
\end{proof}

The Lebesgue measure can be used to define an integral, called the Lebesgue integral (for the Lebesgue measure). The construction of this integral, a standard procedure, will not be covered. One also obtains Lebesgue densities, which will capture the notion of media defined by locally integrable functions and represent the state space of the dynamical systems studied in \cref{sec:conservation-laws}.

\subsection{Spatial Derivatives and Integration by Parts}

Some of the key uses of derivatives of mathematical objects are linear approximation and the representation of (physical) structures through differential equations. It is essential for this work to have a notion of differentiability for networks.\\

The common Euclidean theory of differentiation does not apply to networks as they are abstract metric spaces with non-oriented edges and vertices as a non-trivial boundary. Therefore it is tough to find a general notion of differentiability. The notions of differentiability differ depending on how much one likes to think of vertices as a boundary. The following definition may be the best choice since it enforces continuity constraints on values a function takes on vertices. Its offer of existence of continuous differentiability may be essential to many use-cases.

\begin{definition}[Differentiability]
    A function $f: \NN \to \R$ on a regular network is called \textbf{k-(continuously) differentiable} if $f\circ p$ is k-times (continuously) differentiable for all isometric paths $p: [0, l] \to \NN$.
\end{definition}

This notion of differentiability requires differentiability along all \enquote{structured} paths, yet does not require a notion of orientation, as only for derivatives one needs to pick a (local) orientation. Next to differentiability along edges, the definition turns out to impose an additional constraint of zero-derivatives on all vertices with more than two edges and differentiability across vertices with two edges (see \cref{fig:differentiability}).\\

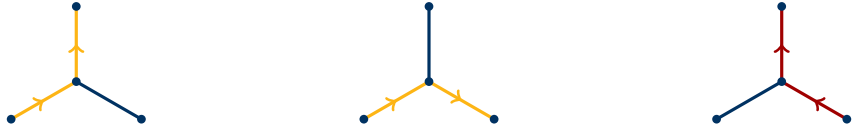
\begin{figure}[ht]
\label{fig:vertexdifferentiability}
\centering
    \makebox[.3\textwidth][c]{
    \begin{tikzpicture}
        \draw[very thick,color=myGold,->-=0.5] (0,0) -- (0,1) ;
        \draw[very thick,color=myGold,->-=0.5] (-0.866,-0.5) -- (0,0);
        \draw[very thick,color=myBlue] (0,0) -- (0.866,-0.5);
        \node at (0,0) [circle,fill=myBlue,scale=.35] {};
        \node at (0,1) [circle,fill=myBlue,scale=.35] {};
        \node at (-0.866,-0.5) [circle,fill=myBlue,scale=.35] {};
        \node at (0.866,-0.5) [circle,fill=myBlue,scale=.35] {};
    \end{tikzpicture}
    }
    \makebox[.3\textwidth][c]{
    \begin{tikzpicture}
        \draw[very thick,color=myBlue] (0,0) -- (0,1) ;
        \draw[very thick,color=myGold,->-=0.5] (-0.866,-0.5) -- (0,0);
        \draw[very thick,color=myGold,->-=0.5] (0,0) -- (0.866,-0.5);
        \node at (0,0) [circle,fill=myBlue,scale=.35] {};
        \node at (0,1) [circle,fill=myBlue,scale=.35] {};
        \node at (-0.866,-0.5) [circle,fill=myBlue,scale=.35] {};
        \node at (0.866,-0.5) [circle,fill=myBlue,scale=.35] {};
    \end{tikzpicture}
    }
    \makebox[.3\textwidth][c]{
    \begin{tikzpicture}
        \draw[very thick,color=goodred,->-=0.5] (0,0) -- (0,1) ;
        \draw[very thick,color=myBlue] (-0.866,-0.5) -- (0,0);
        \draw[very thick,color=goodred,->-=0.5] (0.866,-0.5) -- (0,0);
        \node at (0,0) [circle,fill=myBlue,scale=.35] {};
        \node at (0,1) [circle,fill=myBlue,scale=.35] {};
        \node at (-0.866,-0.5) [circle,fill=myBlue,scale=.35] {};
        \node at (0.866,-0.5) [circle,fill=myBlue,scale=.35] {};
    \end{tikzpicture}
    }
     \caption{Differentiability across vertices with more than two edges: In vertices, derivatives along isometric paths must be zero. Assuming the derivatives of a function along the yellow paths have the same non-zero sign, it can not be differentiated along the red path.}
\label{fig:differentiability}
\end{figure}
Differentiable functions are still missing a derivative. For a derivative, one has to pick an orientation for the edges. The orientation of network edges can be adapted to the orientation of a graph's edges that encodes the network, but it does not have to be. A problem that may appear is that there are vertices with two edges for which one picks different orientations. The following global object can solve this problem.

\begin{definition}[Spatial derivative]
    An operator $D: C^1(\NN) \to C(\NN)$ is called \textbf{spatial derivative} if it resembles a classical derivative, i.e., if for all open $U \subseteq \NN$ for which there is an isometric isomorphism $p: (0, l) \to U$ it holds that for all $f \in C^1(\NN)$
    \[
        (D f)_{|U} \circ p = (f_{|U}\circ p)' \,.
    \]
\label{def:spatialderivative}%
\end{definition}%
\vspace{-7.5mm}
\begin{remark}\leavevmode
    \begin{enumerate}
        \item In the setting of dynamical systems, the notation $(\cdot)_x := \tfrac{\dd}{\dd x} := D$ will also be used.
        \item Each spatial derivative provides a specific assignment of $0_e$ and $w_e$ for each edge $e \in \mathcal{E}$, as each edge inherits an \textbf{orientation} that is embedded into the definition. This fact can be characterized by monotonicity of functions: A function $f \in C^1(\NN)$ is said to be strictly monotonous on $e$ if $(Df)_{|e} > 0$. Then the definition
        \[
        0_e := \argmin f_{|e} \qquad \text{and} \qquad w_e := \argmax f_{|e}
        \]
        does not depend on the choice of $f$.
        \item Given $D$, one can define the \textbf{incoming and the outgoing edges} of vertex $v$ as
        \[
            \mathcal{E}_v^{in} := \{ e \in \mathcal{E} \mid w_e = v\} \qquad \text{and} \qquad \mathcal{E}_v^{out} := \{ e \in \mathcal{E} \mid 0_e = v\} \,.
        \]
        \item A spatial derivative can be computed by a simple reduction to the real-valued case. Indeed, for all $e \in \mathcal{E}$ there exists a unique path $p$ parametrized by length $0_e \rightsquigarrow w_e$. As $p$ is locally an isometric isomorphism, one has for all $x \in e$
        \[
            Df(x) = (f \circ p)'(p^{-1}(x)) \,.
        \]
    \end{enumerate}
\end{remark}
One obtains such a derivative by splitting the network on all vertices with more than two edges and picking an orientation for all parts. The partition's remaining elements are either isometric to intervals, or the original network is (topologically) a sphere. Piecewise derivatives are then compatible as they are zero at vertices with more than two edges, as this is required as per the differentiability of functions.\\

An analog of the classical product rule and the related integration by parts formula also holds for all spatial derivatives of a regular network. These properties are essential when dealing with certain dynamical systems that describe the transport of media in \cref{sec:conservation-laws}.

\begin{lemma}[Product rule]
    A spatial derivative $D: C^1(\NN) \to C(\NN)$ fulfills the product rule, i.e., for all $x \in e \in \mathcal{E}$ and $f,g \in C^1(\NN)$
    \[
        D (fg)(x)
        \equiv
        (fg_{|e})'(x)
        =
        f_{|e}'(x)g_{|e}(x) + f_{|e}(x)g_{|e}'(x)
        \equiv
        D f(x) g(x) + f(x) D g(x) \,.
    \]
\label{lem:productrule}
\end{lemma}
\vspace{-10mm}
\begin{proof}
    Clear.
\end{proof}

\begin{theorem}[Integration by parts]
    Let $D: C^1(\NN) \to C(\NN)$ be a spatial derivative on a regular network $\NN$ that is compact. Then, for all $f,g \in C^1(\NN)$
    \[
        \int_\NN D f g + f D g \, \dd \lambda =%
        \sum_{e \in \mathcal{E}} (fg)(w_e) -%
        (fg)(0_e) = \sum_{v \in \mathcal{V}}%
        \bigg(\sum_{e \in \mathcal{E}_v^{in}}%
        (fg)(w_e) \bigg) -  \bigg(\sum_{e \in%
        \mathcal{E}_v^{out}} (fg)(0_e) \bigg) \,,
    \]
    where $0_e, w_e$ are the endpoints of edge $e$ the order of which depends on $D$.
\label{thm:intbyprts}
\end{theorem}
\begin{proof}
    Let $f,g \in C^1(\NN)$ and $D: C^1(\NN) \to C(\NN)$. One has
    \begin{align*}
        \int_\NN D f g + f D g \, \dd \lambda
        &=
        \sum_{e \in \mathcal{E}} \int_e (fg)_{|e}' \, \dd \lambda_{|e}
    \tag*{\small{(\cref{lem:lebesgue,lem:productrule})}}
        \\
        &=
        \sum_{e \in \mathcal{E}} (fg)(w_e) - (fg)(0_e) \,,
    \tag*{\small{(fundamental theorem)}}
    \end{align*}
    where $0_e, w_e$ are the endpoints of $e$, the order of which depends on $D$. One obtains the second identity in the theorem's statement because each edge is incoming into exactly one vertex and outgoing from exactly one vertex.
\end{proof}

\section{Weak Solutions on Networks and Kirchhoff's First Law}
\label{sec:conservation-laws}

Weak solutions to differential equations can be a generalization of the notion of a solution from continuously differentiable functions to Lebesgue densities. In the case of networks, this generalization manifests as both a generalization to a non-Euclidean setting and a generalization to less regular initial data.\\

In \cref{fig:discontinuity}, it can be seen that even in \enquote{simple} networks, there may not exist a differentiable or even continuous solution to a conservation law. This fact and the possible generalization to less regular initial measures, such as $L^1_{loc}(\NN)$, suggest a weak solution theory for networks. This setting will turn out to yield a coupled system of weak solutions for each edge and thereby recover Kirchhoff's first law from the network's topology.\\

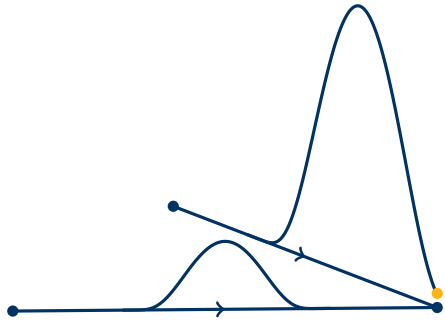
\begin{figure}[ht]
    \centering
    \begin{tikzpicture}
    \begin{axis}[smooth, hide axis,xmin=-3.,xmax=3.,ymin=-3.,ymax=3.,zmin=-.014,zmax=.04,
    height=110mm, width=150mm]
    \addplot3[
            myBlue,
            y domain=0:0,
            very thick,
            samples=60,
            domain=0:0.83,
            ]
        ({.885*(x+0.3)*3},{1.5*(1-.885*(x+0.3))},
        {2*exp(-(1/(x*(1-x))))});
    \addplot3[
            myBlue,
            y domain=0:0,
            very thick,
            samples=60,
            domain=0:1,
            ]
        ({0.5*(x+0.5)*3},{-1.5*(1-0.5*(x+0.5))},
        {.5*exp(-(1/(x*(1-x))))});

    \draw[very thick,->-=0.5,color=myBlue] (axis cs:0,1.5,0) -- (axis cs:3,0,0);
    \draw[very thick,->-=0.5,color=myBlue] (axis cs:0,-1.5,0) -- (axis cs:3,0,0);
    \node at (axis cs:3,0,0) [circle,fill=myBlue,scale=.45] {};
    \node at (axis cs:0,1.5,0) [circle,fill=myBlue,scale=.45] {};
    \node at (axis cs:0,-1.5,0) [circle,fill=myBlue,scale=.45] {};
    
    \node at (axis cs:(3,0,0.0019) [circle,fill=myGold,scale=.45] {};
    \end{axis}
    \end{tikzpicture}
    \caption{With the densities moving to the right on both edges, any choice of value on the right vertex leads to a discontinuity.}
    \label{fig:discontinuity}
\end{figure}

The following conservation law captures the dynamical systems of Lebesgue densities that will be treated in this work.

\begin{remark}
    For functions $\rho: [0, T] \times \NN \to \R$ and $A \subseteq \NN$, the notation $\rho_{|A} := \rho_{|(0, T) \times A}$ will be used. Note that in the following for functions, unless otherwise specified, $(\cdot)_{|\cdot}$ will denote the restriction, $(\cdot)_{\cdot}$ a partial derivative, and $(\cdot)_{\underline{\cdot}}$ the projection onto a factor of the codomain.\\
    For example, for a continuously differentiable function
    \[
        f: \R^2 \to \R^2, (x,y) \mapsto (a(x,y),b(x,y))
    \]
    it will be written
    \begin{gather*}
        f_{\underline{2}} \overset{!}{=} b \,, \quad 
        f_{|[0,1]} \overset{!}{=} ( x\in [0, 1] \mapsto f(x)) \,,\,\, \text{and} \quad
        f_x \overset{!}{=} \frac{\partial}{\partial x} f \,.
    \end{gather*}
    The time integrals $\int \,\cdot \,\, \dd t$ considered will be Lebesgue integrals. Further, $A^\circ$ will denote the interior of $A \subseteq \NN$.
\end{remark}

\begin{definition}[Conservation laws on networks]
    The functions
    \[
        \rho \in C\big([0, T], L^1_{loc}(\NN)\big) \quad \text{and} \quad \nu \in L^\infty_{loc}([0, T] \times \NN) \,,
    \]
    fulfill a \textbf{conservation law} with a spatial derivative $\tfrac{\dd}{\dd x}$ on the regular network $\NN$, if for all $\phi \in C_c^\infty([0, T] \times \NN)$
    \[
        \iint_{[0, T]\times\NN} \rho\,(\phi_t + \nu \phi_x) \, \lambda(\dd x) \, \dd t
        - \int_{\NN} (\rho \phi)(t, \cdot) \big\vert_{t=0}^T \, \dd \lambda = 0 \,.
    \]
\label{def:weaksolution}
\end{definition}

It is important to mention that the above solution concept enforces the flow to be zero on vertices included in only one edge. That is, there is no in- or outflow of the network. These edges may be seen as dead-ins or dead-ends. Excluding them is a purely technical feature that offers a slightly more condensed description. Even in this setting, one can always introduce \enquote{dummy edges} and not lose generalization.\\
To validate the introduced solution concept, it should be made sure it is compatible with the classical setting in the sense of \cref{thm:classical}.

\begin{theorem}[Classical solutions]
    In the setting of \cref{def:weaksolution}, if $\NN$ has only one edge $e$ and $\rho, \nu \in C^1([0, T] \times \NN)$, then $\rho$ fulfills a conservation law with $\nu$ and $\tfrac{\dd}{\dd x}$ if and only if
    \begin{align*}
        & \quad\qquad \rho_t + (\nu \rho)_x \equiv 0\\
        & \text{and if $e$ does not connect a vertex with itself}\\
        & \quad\qquad \nu(t,0_e) \rho(t, 0_e) = \nu(t, w_e) \rho(t, w_e) = 0 \quad \text{for all $t \in [0, T]$} \,.
    \end{align*}%
\label{thm:classical}%
\end{theorem}%
\begin{proof}
    Assume $\rho$ fulfills a conservation law with $\nu$ and $\tfrac{\dd}{\dd x}$.
    Let $\NN$ be a network with only one edge $e$ and $\rho, \nu \in C^1([0, T] \times \NN)$ and $\phi \in C_c^\infty([0, T] \times \NN)$, then
    \begin{align*}
        0
        &=
        \iint_{[0, T]\times\NN} \rho\,(\phi_t + \nu \phi_x) \, \lambda(\dd x) \, \dd t
        - \int_{\NN} (\rho\phi)(t, \cdot) \big\vert_{t=0}^T \, \dd \lambda
        \\
        &=
        \int_{\NN} \bigg( \int_{[0, T]} \rho\phi_t \, \dd t \bigg) -  (\rho\phi)(t, \cdot) \big\vert_{t=0}^T \, \dd \lambda 
        +
        \int_{[0, T]} \int_{\NN} \nu \rho \phi_x \, \lambda(\dd x) \, \dd t
    \tag*{\small{(linearity and Fubini)}}
        \\
        &=
        - \int_{\NN} \int_{[0, T]} \rho_t\phi \, \dd t \, \dd \lambda 
        -
        \int_{[0, T]} \bigg(\int_{\NN} (\nu \rho)_x \phi \, \dd \lambda \bigg)  - (\nu \rho \phi)(\cdot, x) \big\vert_{x = 0_e}^{w_e}\, \dd t
    \tag*{\small{(classical integration by parts and \cref{thm:intbyprts})}}
        \\
        &=
        - \int_{[0, T]} \int_{\NN} (\rho_t + (\nu \rho)_x)\phi \, \lambda(\dd x) \, \dd t
        +
        \int_{[0, T]}  (\nu \rho \phi)(\cdot, x) \big\vert_{x = 0_e}^{w_e}\, \dd t \,.
    \tag*{\small{(Fubini and linearity)}}
    \end{align*}
    The previous equation must hold for all $\phi \in C_c^\infty([0, T] \times \NN)$; in particular if $\phi$ is zero at the domains boundary. Therefore, $\rho_t + (\nu \rho)_x$ must be zero everywhere by the fundamental lemma of the calculus of variations \cite[p.~6,~Lemma~1.1.1]{jost1998}. If $0_e \neq w_e$, then through arbitrary localization on the boundaries $[0, T] \times \{0_e\}$ and $[0, T] \times \{w_e\}$ one obtains by
    \[
        \int_{[0, T]}  (\nu \rho \phi)(\cdot, x) \big\vert_{x = 0_e}^{w_e}\, \dd t = \int_{[0, T]} (\nu \rho)(\cdot, w_e)\,\phi(\cdot, w_e)\, \dd t - \int_{[0, T]}  (\nu \rho )(\cdot, 0_e)\,\phi(\cdot, 0_e)\, \dd t = 0
    \]
    and the fundamental lemma of the calculus of variations that
    \[
        (\nu \rho)(t, 0_e) = (\nu \rho)(t, w_e) = 0 \quad \text{for all $t \in [0, T]$} \,.
    \]
    The opposite direction of the statement is trivial.
\end{proof}

One of the central ideas towards simulating and optimizing network flows is their characterization as flows on edges coupled by a conservation of their boundary flows at mutual vertices. This reduces the presented abstract setting to a more classical one that will turn out to lend itself well for computation. To this end, \cref{thm:kirchhoff} can be understood to recover Kirchhoff's first law \cite{kirchhoff1847} for conservation laws of \cref{def:weaksolution}. To get a sense of the meaning of boundary evaluations of $L^1$-functions, the following lemma is derived in preparation of \cref{thm:kirchhoff}.

\begin{lemma}[Boundary evaluations]
\label{lem:boundaryev}
    Let
    \begin{itemize}
        \item $f \in L^1\big([0, T] \times (0, w)\big)\cap C_u\big((0, w), L^1([0, T])\big)$ for some $w > 0$,
        \item $\tau \in C^\infty_c([0, T])$,
        \item $\tau_\gamma \in C^\infty_c\big([0, T] \times (0, \tfrac{1}{\gamma})\big)$ for all $\gamma \in (\frac{1}{w}, \infty)$ with
        \begin{equation}
            \int_{(0,w)} \tau_\gamma(\cdot, x) \, \dd x  = \tau(\cdot)
        \label{eq:boundaryev:identity}
        \end{equation}
        and such that there exists $M \in \R$ with
        \begin{equation}
            \int_{(0, w)} \norm{\tau_\gamma (\cdot, x)}_{L^\infty([0, T])} \, \dd x < M  \,.
        \label{eq:boundaryev:bound}
        \end{equation}
    \end{itemize}
    By completeness define $f_0 := \lim_{x \to 0} f(\cdot, x) \in L^1([0, T])$. Then,
    \[
        \int_{[0, T]} f_0(t) \tau(t) \, \dd t =  \lim_{\gamma \to \infty} \iint_{[0, T]\times (0,w)} f(t,x) \tau_\gamma(t,x) \, \dd x \, \dd t \,.
    \]
\end{lemma}
\begin{remark}
    $C_u$ denotes uniformly continuous functions. In the following, the limit
    \[
        f(\cdot, 0) := \lim_{x \to 0} f(\cdot, x) \in L^1([0, T])
    \]
    will be heavily used in \textbf{boundary evaluations} of functions $f$ that fulfill the regularity assumptions of \cref{lem:boundaryev}. Note that such limits may differ in network vertices based on the edge in which the limit is considered.
\end{remark}
\begin{proof}
    Let $\varepsilon > 0$ and pick $\alpha > \tfrac{1}{w}$ such that for all $x < \tfrac{1}{\alpha}$
    \[
        \norm{f_0(\cdot)-f(\cdot,x)}_{L^1([0,T])} < \varepsilon \,.
    \]
    One has for all $\gamma \in (\alpha, \infty)$
    \begin{align*}
        &\bigg|\int_{[0, T]} f_0(t) \tau(t) \, \dd t - \iint_{[0, T]\times (0,w)} f(t,x) \tau_\gamma(t,x) \, \dd x \, \dd t\bigg|
        \\
        &\quad\qquad=
        \bigg| \iint_{[0, T]\times (0,w)}  f_0(t) \tau_\gamma(\cdot, x) \, \dd x \, \dd t - \iint_{[0, T]\times (0,w)} f(t,x) \tau_\gamma(t,x) \, \dd x \, \dd t\bigg|
    \tag*{\small{(assumption, linearity)}}
        \\
        &\quad\qquad=
        \bigg| \iint_{[0, T]\times (0,w)}  (f_0(t)-f(t,x)) \tau_\gamma(\cdot, x) \, \dd x \, \dd t\bigg|
    \tag*{\small{(linearity)}}
        \\
        &\quad\qquad\leq
        \iint_{(0,w) \times [0, T]} \big| (f_0(t)-f(t,x)) \tau_\gamma(\cdot, x) \big|\, \dd t \, \dd x
    \tag*{\small{(Fubini's theorem, triangle inequality)}}
        \\
        &\quad\qquad=
        \int_{(0,1/\gamma)} \norm{f_0(\cdot)-f(\cdot,x)}_{L^1([0,T])} \norm{\tau_\gamma(\cdot, x)}_{L^\infty([0, T])}  \, \dd x
    \tag*{\small{(Hölder's inequality)}}
        \\
        &\quad\qquad\leq
        \varepsilon M \,.
    \tag*{\small{(assumptions)}}
    \end{align*}
\end{proof}

\begin{theorem}[Kirchhoff's first law]
    In the setting of \cref{def:weaksolution}, if for all edges $e \in \mathcal{E}$
    \begin{equation}
        (\rho\nu)_{|e} \in C_u\big(e^o, L^1([0,T])\big)
    \label[property]{eq:thm:kirchhoff:regularity}
    \end{equation}
    then $\rho$ fulfills a conservation law with $\nu$ and $\tfrac{\dd}{\dd x}$ if and only if
    \begin{align*}
    &
    \text{$\rho$ is a weak solution on edges, i.e.,}
    \\
    &\qquad
    \text{for all edges $e \in \mathcal{E}$ and for all $\phi \in C^\infty_c([0,T]\times e^\circ)$}
    \\
    & \quad\qquad
    \iint_{[0, T]\times e} \rho\,(\phi_t + \nu \phi_x) \, \dd \lambda_{|e} \, \dd t
        - \int_{e} (\rho \phi)(t, \cdot) \big\vert_{t=0}^T \, \dd \lambda_{|e}    \, = 0 \,,
    \\
    &
    \text{and the flow is conserved at vertices, i.e.,}
    \\
    &\qquad
    \text{for all vertices $v \in \mathcal{V}$ and for almost every $t \in [0, T]$}
    \\
    & \quad\qquad
    \bigg(\sum_{e \in \mathcal{E}^{in}_v} (\nu \rho)_{|e}(t, w_e)\bigg) = \bigg(\sum_{e \in \mathcal{E}^{out}_v} (\nu \rho)_{|e}(t, 0_e)\bigg) \,.
    \end{align*}
\label{thm:kirchhoff}
\end{theorem}
\begin{remark}
\label{rem:kirchhoff}
    It is essential to note that \cref{thm:kirchhoff} does not make a statement on the existence or uniqueness of solutions. The solution concept at hand relaxes the widely known setting of the existence of a solution $\rho$ for a given $\rho \mapsto \Tilde{\nu}(\rho)$ to the assumption of the existence of a solution tuple $(\nu, \rho)$. In particular, there may be a representation $\nu \equiv \Tilde{\nu}(\rho)$. Therefore, the statement of \cref{thm:kirchhoff} applies in greater generality, in particular, to hyperbolic conservation laws that admit a quasi-linear form. Generalizations to vector-valued conservation laws modeling multi-commodity flows as well as to balance laws are possible. An introduction to hyperbolic conservation laws can be found in \cite{bressan2000}.
\end{remark}
\begin{proof}
The proof is broken down into several steps.
\begin{itemize}
    \item Let $\rho$ fulfill a conservation law with $\nu$ and $\tfrac{\dd}{\dd x}$.
    Then by extension of test functions from edges to the network with $0$, it is clear that $\rho$ is a weak solution on edges.
    
    \item To show that the flow is conserved at vertices let $v \in \mathcal{V}$ and $\phi \in C_c^\infty([0, T] \times \NN)$ such that the support of $\phi$ is contained in an $\varepsilon$-ball $U_\varepsilon(v)$ around $v$ at all times, where $2\varepsilon$ is smaller than the smallest edge length, i.e.,
    \[
        \mathrm{supp} \, \phi \subseteq [0,T] \times U_\varepsilon(v)\,.
    \]
    Note that $U_\varepsilon(v)$ can be contracted with a continuous bijection $c_\gamma: U_\varepsilon(v) \to U_{\varepsilon/\gamma}(v)$ uniquely, such that,
    \[
        d(v, c_\gamma(\cdot)) = d(v, \cdot)/\gamma \quad \text{for all} \,\, \gamma \geq 1 \quad  \text{and} \quad c_1 = \mathrm{id} \,.
    \]
    For all $\gamma \geq 1$ define $\phi^\gamma: [0, T] \times \NN \to \R$ by
    \[
        \phi^\gamma(t, x) :=
        \begin{cases}
            \phi(t, c_\gamma^{-1}(x)) & \text{if} \,\, x \in U_{\varepsilon/\gamma}(v)
            \\
            0 & \text{otherwise.}
        \end{cases}
    \]
    \end{itemize}
    Generally, as $\rho$ fulfills a conservation law with $\nu$ and $\tfrac{\dd}{\dd x}$
    \begin{equation}
        0
        =
        \iint_{[0, T] \times \NN} \rho \phi^\gamma_t \dd \lambda \, \dd t
        +
        \iint_{[0, T]\times\NN} \nu \rho \phi^\gamma_x \, \dd \lambda \, \dd t
        -
        \int_{\NN} (\rho \phi)(t, \cdot) \big\vert_{t=0}^T \, \dd \lambda  \,.
    \label{eq:thm:kirchhoff:conteq}
    \end{equation}
    \begin{itemize}
    \item Therefore, by the fact that
    \[
        \mathrm{supp} \, \phi^\gamma \subseteq [0, T] \times U_{\varepsilon/\gamma}(v) \,,
    \]
     with $n := \abs{\mathcal{E}_v}$ being the number of edges connected to $v$, and the latter factor $2$ due to possible self-edges, it follows that
    \begin{equation}
        \int_{\NN} (\rho \phi)(t, \cdot) \big\vert_{t=0}^T \, \dd \lambda \leq 2\norm{\rho(0, \cdot)_{|\mathrm{supp} \, \phi}}_{L^1} \abs{\max \phi} \frac{\varepsilon}{\gamma} 2n \xrightarrow{\gamma \to \infty} 0 \,.
    \label{eq:thm:kirchhoff:shrnktm1}
    \end{equation}
    By
    \[
        \mathrm{supp} \, \phi_t^\gamma \subseteq [0, T] \times U_{\varepsilon/\gamma}(v)
    \]
    one has
    \begin{equation}
        \iint_{[0, T]\times \NN} \rho\phi^\gamma_t \, \dd \lambda \, \dd t \leq \norm{\rho_{|\mathrm{supp} \, \phi}}_{L^1} \abs{\max \phi_t} \frac{\varepsilon}{\gamma} T 2n \xrightarrow{\gamma \to \infty} 0 \,.
    \label{eq:thm:kirchhoff:shrnktm2}
    \end{equation}
    
    \item One therefore has
    \begin{align*}
        0
        &=
        \lim_{\gamma \to \infty} \iint_{[0, T]\times\NN}  \nu \rho\phi^\gamma_x \, \dd \lambda \, \dd t
    \tag*{\small{(\cref{eq:thm:kirchhoff:conteq,eq:thm:kirchhoff:shrnktm1,eq:thm:kirchhoff:shrnktm2})}}
        \\
        &= \lim_{\gamma \to \infty} \sum_{e \in \mathcal{E}_v} \iint_{[0, T]\times e^\circ}  (\nu \rho)_{|e}{\phi^\gamma_x}_{|e} \, \dd \lambda_{|e} \, \dd t
    \tag*{\small{(linearity)}}
        \\
        &=
        \int_{[0, T]} \phi(\cdot, v) \bigg(\bigg(\sum_{e \in \mathcal{E}^{in}_v} (\nu \rho)_{|e}(t, w_e)\bigg) - \bigg(\sum_{e \in \mathcal{E}^{out}_v} (\nu \rho)_{|e}(t, 0_e) \bigg)\bigg)\, \dd t
    \tag*{\small{(\cref{lem:boundaryev}, see below argument)}}
        \\
        \implies \quad &
        \bigg(\sum_{e \in \mathcal{E}^{in}_v} (\nu \rho)_{|e}(t, w_e)\bigg) = \bigg(\sum_{e \in \mathcal{E}^{out}_v} (\nu \rho)_{|e}(t, 0_e) \bigg)
        \quad \text{for almost every $t \in [0, T]$.}
    \tag*{\small{($\phi(\cdot, v)$ was arbitrary, fundamental lemma c.v. \cite[p.~6,~Lemma~1.1.1]{jost1998})}}
    \end{align*}
    To see why \cref{lem:boundaryev} is applicable for all $e \in \mathcal{E}_v$, without loss of generalization, assume $e \in \mathcal{E}_v^{out}$; if $0_e = w_e$, the two halfs of the edge need to be treated separately. In the notation of \cref{lem:boundaryev} let
    \[
        f \overset{!}{=} (\nu \rho)_{|e} \,, \quad \tau \overset{!}{=} \phi(\cdot, v) \,, \,\, \text{and} \quad \tau_\gamma \overset{!}{=} {\phi_x^\gamma}_{|e} \,.
    \]
    Then, by \cref{eq:thm:kirchhoff:regularity} and the fact that
    \[
        \mathrm{supp}\, {\phi_x^\gamma}_{|e} \subseteq [0, T] \times [0, \varepsilon/\gamma)\,,
    \]
    it remains to show that \cref{eq:boundaryev:identity,eq:boundaryev:bound} hold.
    By \cref{thm:intbyprts} and the definition of $\phi^\gamma$, \cref{eq:boundaryev:identity} holds, i.e., for all $t \in [0, T]$
    \[
        \int_e {\phi_x^\gamma}_{|e}(t,\cdot) \, \dd \lambda_{|e} = \phi(t, v) \,.
    \]
    Again by definition of $\phi^\gamma$, \cref{eq:boundaryev:bound} holds, i.e.,
    \[
        \int_{e} \big\lVert{\phi_x^\gamma}_{|e}(\cdot, x)\big\rVert_{L^\infty([0, T])} \, \lambda_{|e}(\dd x) 
        \leq
        \int_{(0, \varepsilon/\gamma)} \gamma \norm{\phi_x}_{L^\infty([0,T]\times e)} \, \dd x
        = \varepsilon  \norm{\phi_x}_{L^\infty([0,T]\times e)} \,.
    \tag*{\small{(by ${\phi_x^\gamma}_{|e} \leq \gamma \norm{\phi_x}_{L^\infty([0,T]\times e)}$)}}
    \]

    \item The converse statement of this theorem holds by the following argument. Assume that $\rho$ is a weak solution on edges with flow conservation at all vertices. It is the goal to show that
    for all $\psi \in C_c^\infty([0, T] \times \NN)$
    \begin{equation}
        \int_{[0, T]} \int_{\NN} \rho\,(\psi_t + \nu \psi_x) \, \dd \lambda \, \dd t
        - \int_{\NN}  (\rho \psi)(t, \cdot) \big\vert_{t=0}^T\, \dd \lambda = 0 \,.
    \label{eq:thm:kirchhoff:weak}
    \end{equation}
    It is equivalent to verify \cref{eq:thm:kirchhoff:weak} for test functions of type $\psi \phi^\gamma$, where $\phi^\gamma$ is a localized function, such that,
    \[
        0 \leq \phi^\gamma \leq 1 \,,
    \]
    and for a vertex $v \in \mathcal{V}$ and $\varepsilon>0$ (as above)
    \[
        \phi^\gamma\big(\cdot, U_{\varepsilon/2\gamma}(v)\big) \equiv 1 \quad \text{and} \quad \phi^\gamma\big(\cdot, U_{\varepsilon/\gamma}(v)^C\big) \equiv 0 \,.
    \]
    This is because $\phi^\gamma$ can be part of a partition of unity such that all other terms are zero by the assumption of $\rho$ being a weak solution on open edges. Therefore, the value of
    \[
        \iint_{[0, T]\times\NN} \rho\, ((\psi \phi^\gamma)_t + \nu \, (\psi\phi^\gamma)_x) \, \dd \lambda \, \dd t
        - \int_{\NN} (\rho \psi\phi^\gamma)(t, \cdot) \big\vert_{t=0}^T  \, \dd \lambda
    \]
    is constant in $\gamma>1$ and converges as $\gamma \to \infty$\,---by the exact analog derivation as above ---to the difference of incoming and outgoing flow, which is zero by assumption.
\end{itemize}
\end{proof}

\section{Future Work}
We see most promising focus of future work building on the described theory in
\begin{description}
    \item[Robust solutions to hyperbolic conservation laws] The setting of Polish spaces suggests the application of existing tools from probability to meaningfully model scenarios,
    \item[Convergence in time] Stable solutions to physically-plausible hyperbolic conservation laws on networks at long time horizons, and
    \item[Localization in networks] Uncertain measurements of characteristics of hyperbolic conservation laws on networks embedded into a higher-level physical world model.
\end{description}